\documentclass[12pt]{article}

\usepackage{amsmath, amsthm, amssymb}
\usepackage{hyperref}
\usepackage{geometry}
\usepackage{enumitem}
\usepackage{cite}
\usepackage{microtype}
\usepackage{ragged2e}

\newtheorem{theorem}{Theorem}[section]
\newtheorem{lemma}[theorem]{Lemma}
\newtheorem{proposition}[theorem]{Proposition}
\newtheorem{corollary}[theorem]{Corollary}

\theoremstyle{definition}

\newtheorem{example}[theorem]{Example}
\newtheorem{remark}[theorem]{Remark}
\newtheorem{observation}[theorem]{Observation}
\newtheorem{convention}[theorem]{Convention}

\newcommand{\ex}{\mathrm{ex}}

\title{\textbf{Bipartite Tur\'{a}n Numbers of Trees and Star Forests}}

\author{
	Omid Khormali\\
	\small Department of Mathematics, Harvey Mudd College, Claremont, CA 91711\\
	\small \texttt{okhormali@hmc.edu}
}

\date{}
\begin{document}
	\sloppy
	
	\maketitle
	
	\begin{abstract}
		
		
		The bipartite Tur\'an number of a graph $H$, denoted $\text{ex}(m, n; H)$, is the maximum
		number of edges in any $H$-free bipartite graph $G = (A, B; E)$ with parts of size $|A| = m$ and $|B| = n$. We study this problem for two families. For a tree $T = T(r, s)$ with parts   $R$ and $S$ of sizes $|R| = r \le s = |S|$, we prove
		\[
		(r - 1) n \;\le\; \text{ex}(m, n; T(r, s)) \;\le\; (r - 1) n + C(m,r,s)
		\]
		for $n$ sufficiently large compared to $m$, $r$, and $s$, where
		$C(m,r,s)$ does not depend on $n$, determining the leading-order term
		exactly (with the star case $r = 1$ solved with an exact formula). For a star forest $F = \bigcup_{i=1}^k S_{d_i}$ with $d_1 \ge \cdots \ge d_k$, we determine the exact value $\text{ex}(m, n; F) = (k - 1) n + (d_k - 1)(m - k + 1)$ for $n$ sufficiently large, and characterize the unique extremal graph.

	\end{abstract}
	

	\bigskip

	\section{Introduction}
	
	The \emph{Tur\'{a}n number} of a graph $H$, denoted $\ex(n, H)$, is the maximum number of edges in an $H$-free graph on $n$ vertices. Determining Tur\'{a}n numbers is one of the central problems of extremal graph theory, initiated by the classical theorem of
	Tur\'{a}n~\cite{Turan1941} on complete graphs and extended to general graphs by the Erd\H{o}s--Stone--Simonovits theorem~\cite{ErdosStone1946, ErdosSimonovits1966}, which
	states that
	\[
	\ex(n, H) = \left(1 - \frac{1}{\chi(H)-1}\right)\frac{n^2}{2} + o(n^2),
	\]
	where $\chi(H)$ is the chromatic number of $H$. For bipartite graphs $H$ (i.e., $\chi(H) = 2$), this theorem only gives $\ex(n, H) = o(n^2)$, and determining the correct order of magnitude becomes significantly harder.\\
	
	A natural variant is the \emph{bipartite Tur\'{a}n number}. Given a
	graph $H$ and integers $m \leq n$, we define
	\[
	\ex(m, n;\, H)
	\]
	to be the maximum number of edges in any $H$-free bipartite graph $G = (A, B; E)$ with $|A| = m$ and $|B| = n$. This problem is closely interconnected with the standard Tur\'{a}n problem~\cite{Furedi1996, KST1954}. Classic results in this setting include the K\H{o}v\'{a}ri--S\'{o}s--Tur\'{a}n theorem~\cite{KST1954}
	for complete bipartite graphs,
	\[
	\ex(m, n;\, K_{s,t}) \leq \frac{1}{2}(s-1)^{1/t}\,n\,m^{1-1/t} + \frac{t-1}{2}\,m,
	\]
	and the result of Gy\'{a}rf\'{a}s, Rousseau, and  Schelp~\cite{Gyarfas1984} which completely determined $\ex(m, n;\, P_k)$ for all values of $m$, $n$, and $k$, where $P_k$ denotes the path on $k$ vertices. More recently, Chen, Wang, Yuan, and
	Zhang~\cite{Chen2022} extended this to linear forests, which is vertex-disjoint unions of paths, in the bipartite host setting, thereby establishing the bipartite analogue of the Tur\'{a}n numbers determined by Lidick\'{y}, Liu, and Palmer~\cite{LLP2013} for general hosts.\\
	
	Trees occupy a special place in the Tur\'{a}n theory of sparse graphs. Erd\H{o}s and Gallai~\cite{ErdosGallai1959} determined the Tur\'{a}n number of paths.
	
	\begin{theorem}[Erd\H{o}s--Gallai~\cite{ErdosGallai1959}]\label{thm:EG}
		For any $k, n \geq 2$,
		\[
		\ex(n, P_k) \leq \frac{k-2}{2}\,n,
		\]
		with equality achieved by the disjoint union of complete graphs $K_{k-1}$.
	\end{theorem}
	
	A far-reaching generalization is the \emph{Erd\H{o}s--S\'{o}s conjecture}~\cite{Erdos1964}, which asserts that every graph on $n$ vertices with more than $\frac{k-2}{2}n$ edges contains every tree on $k$ vertices as a subgraph. Although still open in full generality, the conjecture has been verified for many special classes of trees and for
	large trees by Ajtai, Koml\'{o}s, Simonovits, and Szemer\'{e}di (announced but not yet published in full).\\
	
	In the bipartite host setting, the Erd\H{o}s--S\'{o}s conjecture takes a natural form that for a given tree $T$ on $k$ vertices, what is the threshold edge density in a bipartite graph $G[A,B]$ with $|A|=m$, $|B|=n$ that forces $T \subseteq G$? This question has received comparatively less attention than the general host version, and our paper contributes to this line of research.\\

	Let $T = T(r,s)$ denote a tree with parts $R$ and $S$ where $|R| = r$ and $|S| = s$. We always assume without loss of generality that $r \leq s$, so that $R$ is the smaller part. The following result gives an upper bound on $\ex(m, n;\, T(r,s))$ in terms of $r =
	\min(r,s)$ and $n$.

	
	
	\begin{theorem}\label{thm:main}
		Let $T = T(r,s)$ be a tree with parts $R$ and $S$ such that 
		$|R| = r \leq s = |S|$. If $r = 1$ then 
		$\ex(m,n;\,T(1,s)) = (s-1)m$. If $r \geq 2$, let $m$ and $n$ 
		be positive integers satisfying $n \geq s(m-r+2)\binom{m}{r-1}$.
		Then
		\[
		\ex(m,\,n;\,T(r,s)) \;\leq\; (r-1)\,n \;+\; C(m,r,s),
		\]
		where $C(m,r,s) = \binom{m}{r-1}^2(s-1) + m\lfloor 
		\frac{m}{r-1}\rfloor s + ms(m-r+2)\binom{m}{r-1}$.
	\end{theorem}
	
	\bigskip 
	
	Consider the bipartite graph $G = (A, B; E)$ where $|A| = m$, $|B| = n$, $r\geq 2$,
	and $G$ is obtained by taking $r-1$ vertices of $A$ each adjacent to all
	vertices of $B$, with the remaining $m - (r-1)$ vertices of $A$ isolated. Equivalently, $G = K_{r-1, n} \cup \overline{K}_{m-r+1}$. Then $e(G) = (r-1)n$ and $G$ contains no copy of $T(r,s)$, since any embedding of $T$ requires at least $r$ vertices on the $A$-side. Thus we have the following corollary.
	
	\begin{corollary}
		Let $T = T(r,s)$ be a tree with $r = \min(|R|,|S|) \ge 2$. For $m \ge r$
		and $n$ sufficiently large compared to $m,r,s$:
		$$(r-1)n \;\le\; \ex(m,n;T(r,s)) \;\le\; (r-1)n + C(m,r,s),$$
		where $C(m,r,s)$ is the constant of Theorem~1.2, polynomial in $m$ of
		degree depending on $r$. For $r=1$, Theorem~1.2 already gives the exact
		value $\ex(m,n;T(1,s)) = (s-1)m$.
	\end{corollary}

	\vspace{.1in}
	The Erd\H{o}s--S\'{o}s conjecture~\cite{Erdos1964} asserts that 
	$\ex(n, T) \leq \frac{k-2}{2}n$ for any tree $T$ on $k$ vertices, where the host graph is a general graph on $n$ vertices. Our setting is fundamentally  different that the host graph is required to be bipartite. Since the classes of admissible host graphs are distinct, the two bounds $\ex(n,T)$ and $\ex(m,n;\,T)$ are not directly comparable, and we make no claim that our result implies or is implied by the Erd\H{o}s--S\'{o}s conjecture. Rather, our result can be seen as the natural analogue of that conjecture in the bipartite host setting.
	
	\begin{example}
		Let $T = P_k$ be the path on $k$ vertices. If $k$ is even 
		then $r = s = k/2$, and if $k$ is odd then $r = \lfloor k/2 
		\rfloor$ and $s = \lceil k/2 \rceil$. In both cases $r = 
		\lfloor k/2 \rfloor \geq 2$ for $k \geq 4$, and our bound 
		gives
		\[
		\ex(m,\, n;\, P_k) \;\leq\; 
		\bigl(\lfloor k/2 \rfloor - 1\bigr)\,n \;+\; C(m,r,s),
		\]
		where $C(m,r,s)$ is the constant of Theorem~\ref{thm:main}.
		The exact value was determined by Gy\'{a}rf\'{a}s, Rousseau, 
		and Schelp~\cite{Gyarfas1984}: for $n$ sufficiently large 
		compared to $m$,
		\[
		\ex(m,\, n;\, P_k) \;=\; 
		\bigl(\lfloor k/2 \rfloor - 1\bigr)(m + n - 1) + O(1),
		\]
		which for $n \gg m$ is approximately $(\lfloor k/2 \rfloor 
		- 1)n$. This confirms that the leading term $(r-1)n$ in 
		Theorem~\ref{thm:main} is tight, and the error term 
		$C(m,r,s)$ is a constant in $n$ as expected. This is 
		consistent with the gap $(r-1)n \leq \ex(m,n;\,T) \leq 
		(r-1)n + C(m,r,s)$ of Corollary~\ref{cor:bounds}.
	\end{example}
	
	
	We close this discussion by comparing Theorem~1.2 with a closely related result obtained independently and concurrently by Waite and Aydin \cite{WaiteAydin}, stated below in our notation.
	
	\begin{theorem}[Waite--Aydin \cite{WaiteAydin}] \label{thm:waite-aydin}
		Let $T_{a,b} = (A,B)$ be a tree with $|A| = a \ge b = |B|$, and let $\sigma(T_{a,b})$ denote the maximum, over vertices $x$ in the smaller part $B$, of the number of neighbors of $x$ that are either leaves or are adjacent only to leaves besides $x$. If
		\[
		\sigma(T_{a,b}) \ge \min\{a-b,\, b-1\},
		\]
		then for all $m \ge b-1$ and $n \ge a-1$,
		\[
		ex(m,n;T_{a,b}) \le (b-1)n +
		\begin{cases}
			(a-b)m & \text{if } a \ge 2b-1, \\[2pt]
			(b-1)m & \text{if } a \le 2b-1.
		\end{cases}
		\]
	\end{theorem}
	
	Independently, Waite and Aydin [13] study the closely related problem
	of bipartite extremal numbers for trees and obtain Theorem 1.5 above,
	which is sharper than our Corollary 1.3 whenever it applies. Their
	secondary term is linear in $m$ with an explicit constant, whereas our
	constant $C(m,r,s)$ grows polynomially in $m$ (of degree depending on
	$r$); however, their threshold relating $m$ and $n$ is mild
	($n \ge a-1$), compared to the polynomial threshold
	$n \ge s(m-r+2)\binom{m}{r-1}$ required in Theorem~1.2. However, their
	result requires the structural condition $\sigma(T_{a,b}) \ge
	\min\{a-b, b-1\}$, which restricts their theorem to a subclass of trees, including stars, paths, double-stars, brooms, trees whose part-sizes differ by at most one, and all trees on at most seven vertices. Our Theorem~1.2 places no restriction on the shape of $T(r,s)$ and thus determines the leading-order coefficient $(r-1)$ for every tree once $n$ is sufficiently large relative to $m$, $r$, and $s$, including trees outside the scope of their embedding technique. For instance, it applies to the tree obtained by attaching a leaf to the center vertex of a path on seven vertices, which Waite and Aydin identify as the smallest tree not covered by their method. The two results also rely on different techniques. Our proof proceeds by iteratively extracting common neighborhoods in the style of K\H{o}v\'ari--S\'os--Tur\'an, while Waite and Aydin use a weighted generalization of the classical minimal-subgraph method. We view the two results as complementary. Theorem~\ref{thm:waite-aydin} gives sharper, conjecturally tight constants on a broad class of trees under a mild threshold on $n$, while our Theorem~1.2 determines the leading term for all trees once $n$ is sufficiently large, thereby resolving the leading-order case of Waite and Aydin's Conjecture~1.3 (fixed-part form) for such trees in the large-$n$ regime, up to the secondary term. Determining the correct leading-order behavior for smaller $n$, where the linear threshold of Theorem~\ref{thm:waite-aydin} may hold but ours may not yet apply, remains open. It remains open as well whether the secondary term $O(m)$ in Theorem~1.2 can be sharpened to linear in $m$ for all trees, as their conjecture predicts.

	\bigskip
	In addition, we determine $\ex(m, n; F)$ where $F = \bigcup_{i=1}^{k} S_{d_i}$ is a star forest with components of degrees $d_1 \geq d_2 \geq \cdots \geq d_k \geq 1$, and the host is a bipartite graph $G = (A, B; E)$ with $|A| = m \leq n = |B|$. The {\it extremal construction} for $F$ is as follows. For $1 \leq i \leq k$, define $\mathcal{B}(m, n; i)$ to be the bipartite graph obtained by taking $i - 1$ vertices of $A$ each adjacent to all of $B$, and letting each of the remaining $m - i + 1$ vertices of $A$ have degree exactly $d_i - 1$ into $B$, with their neighborhoods in $B$ chosen pairwise disjoint which holds once $n$ is sufficiently large. In particular, no vertex of $B$ is adjacent to more than one of the $m-i+1$ non-universal vertices. The number of edges is
	\[
	e(\mathcal{B}(m, n; i))
	\;=\; (i-1)\,n \;+\; (d_i - 1)(m - i + 1).
	\]
	Since $n$ is sufficiently large compared to $m$, the dominant term is $(i-1)n$, which is strictly increasing in $i$. Hence the maximum over all $i$ is achieved at $i = k$, giving the candidate extremal graph $\mathcal{B}(m, n; k)$ with
	\[
	e(\mathcal{B}(m, n; k))
	\;=\; (k-1)\,n \;+\; (d_k - 1)(m - k + 1).
	\]
	
	
	Note that $\mathcal{B}(m,n;k)$ is $F$-free. Each non-universal vertex has degree $d_k - 1$, too small to center any $S_{d_i}$. A star centered at $v \in B$ can use at most one non-universal leaf, since their neighborhoods in $B$ are pairwise disjoint, so it needs at least one universal leaf too (when $d_i = 1$, non-universal vertices are isolated, so the same holds). Thus every star of $F$ uses a universal vertex, but embedding all $k$ disjoint stars would need $k$ of them, while only $k-1$ exist. So $F$ cannot embed in $\mathcal{B}(m,n;k)$.
	
	\begin{theorem}\label{thm:starforest}
		Let $F = \bigcup_{i=1}^{k} S_{d_i}$ be a star forest with $d_1 \geq d_2
		\geq \cdots \geq d_k \geq 1$. For $n$ sufficiently large compared to $m$
		and $d_1, \ldots, d_k$,
		\[
		\ex(m, n;\, F)
		\;=\; (k-1)\,n \;+\; (d_k - 1)(m - k + 1).
		\]
		The unique extremal graph is $\mathcal{B}(m, n; k)$.
	\end{theorem}

	The remainder of the paper is organized as follows. In Section~\ref{sec:prelim}, we introduce notation and state the key lemmas. In Section~\ref{sec:results}, we prove Theorems \ref{thm:main} and \ref{thm:starforest}.

	\section{Preliminaries}\label{sec:prelim}

	All graphs in this paper are finite, simple, and undirected. For a graph $G$, we write $V(G)$ for the vertex set, $E(G)$ for the edge set, and $e(G) = |E(G)|$ for the number of edges. For $v \in V(G)$, the \emph{neighborhood} of $v$ is $N_G(v) = \{u : uv \in E(G)\}$ and the \emph{degree} of $v$ is $d_G(v) = |N_G(v)|$. 
	For a set $X \subseteq V(G)$, the \emph{common neighborhood} of $X$ is
	$	N_G[X] = \bigcap\limits_{v \in X} N_G(v).$
	
	For $X, Y \subseteq V(G)$, we write $e(X, Y)$ for the number of edges with one endpoint in $X$ and the other in $Y$. For two vertex-disjoint graphs $G$ and $H$, we write $G \cup H$ for their disjoint union. A \emph{bipartite graph} $G = (A, B; E)$ has vertex set $A \cup B$ with $A \cap B =	\emptyset$ and $E$ is the set of all edges between $A$ and $B$. The complete bipartite graph with parts of sizes $p$ and $q$ is
	denoted $K_{p,q}$.\\
	
	A graph is \emph{connected} if there is a path between every pair of vertices. A \emph{tree} is a specific type of connected graph that contains no cycles. A tree $T$ is bipartite with parts $R$ and $S$, and we write $T = T(r,s)$ to indicate that $|R| = r$ and $|S| = s$ with $r \leq s$. A \emph{subgraph} embedding of $H$ in $G$ is an injective map $\phi: V(H) \to V(G)$ such that $\phi(u)\phi(v) \in E(G)$ whenever $uv \in E(H)$. We say $G$ \emph{contains} $H$ (written $H \subseteq G$) if such an embedding exists. \\
	
	The \emph{independence number} of a graph $G$, denoted $\alpha(G)$, is the maximum size of an independent set in $G$, where an \emph{independent set} is a set of vertices no two of which are adjacent. A \emph{vertex cover} of a graph $H$ is a set $C \subseteq V(H)$ such that every edge of $H$ has at least one endpoint in $C$. The \emph{minimum vertex cover number} $\tau(H)$ is the size of a smallest vertex cover. An independent vertex cover $IC$ is a vertex cover of $H$ which is an independent set. We show the size of the minimum independent vertex cover (MIC) by $\tau_{ind}(H)$. A \emph{matching} in $H$ is a set of pairwise vertex-disjoint edges. The \emph{matching number} $\nu(H)$ is the size of a maximum matching. The following classical theorem of K\H{o}nig is fundamental:
	
	\begin{theorem}[K\H{o}nig's Theorem~\cite{König1931}]\label{thm:konig}
		For any bipartite graph $H$,
		\[
		\tau(H) = \nu(H).
		\]
	\end{theorem}
	
	Also, we have the theorem.
	\begin{theorem}[Gallai~\cite{Gallai1959}]\label{thm:gallai}
		For any graph $G$,
		\[
		\alpha(G) + \tau(G) = |V(G)|.
		\]
	\end{theorem}
	
	The following lemma is due to Chen, Wang, Yuan, and Zhang~\cite{Chen2022}. It is the
	primary tool for passing from an edge density condition to a structural embedding.
	
	\begin{lemma}[Chen--Wang--Yuan--Zhang~\cite{Chen2022}]\label{lem:common_nbhd}
		Let $G = (A, B; E)$ be a bipartite graph with $|A| = a$ and $|B| = n$, where $a$
		is a fixed constant. If $a \geq b$ and $e(G) \geq bn$ for $n$ sufficiently large
		compared to $a$, then there exists a vertex set $A' \subseteq A$ with
		$|A'| = \lceil b \rceil$ and a constant $c > 0$ such that
		\[
		|N_G[A']| \geq cn.
		\]
	\end{lemma}
	
	In other words, if $G$ has linearly many edges in $n$, then some $\lceil b \rceil$ vertices of $A$ share at least $cn$ common neighbors in $B$.

	\section{Proofs of Main Theorems}\label{sec:results}
	
	For a bipartite tree, we have the following quick observation.
	
	%
	%
	%
	
	\begin{observation}\label{obs:indep_cover}
		Let $T = T(r,s)$ be a tree with parts $R$ and $S$, and
		$|R| = r \leq s = |S|$. Then $T$ has a MIC, and
		\[
		\tau(T) \leq \tau_{\mathrm{ind}}(T) \leq r.
		\]
	\end{observation}
	
	\begin{proof}
		Since $T$ is bipartite with parts $R$ and $S$, every edge 
		of $T$ has one endpoint in $R$ and one in $S$. Hence $R$ 
		is an independent set and every edge is incident to a vertex 
		of $R$, so $R$ is an independent vertex cover of $T$ with 
		$|R| = r$. This shows $T$ has a MIC and 
		$\tau_{\mathrm{ind}}(T) \leq r$.
		
		For the inequality $\tau(T) \leq \tau_{\mathrm{ind}}(T)$, note 
		that every independent vertex cover is in particular a vertex 
		cover (it satisfies the vertex cover condition with the additional 
		constraint of being independent). Hence the minimum over all 
		independent vertex covers is at least the minimum over all vertex 
		covers, giving $\tau(T) \leq \tau_{\mathrm{ind}}(T)$.
	\end{proof}
	
	\begin{remark}
		The inequality $\tau(T) \leq \tau_{\mathrm{ind}}(T)$ relies 
		on $T$ being a tree and hence bipartite. For graphs containing 
		odd cycles, an independent vertex cover may not exist at all. 
		For example, in the complete graph $K_3$, any vertex cover 
		requires at least two vertices, but no two vertices of $K_3$ 
		are non-adjacent, so no independent vertex cover exists and 
		$\tau_{\mathrm{ind}}(K_3)$ is undefined. More generally, 
		$\tau_{\mathrm{ind}}(G)$ is well-defined if and only if $G$ 
		is bipartite, since only bipartite graphs have independent 
		sets that can cover all edges.
	\end{remark}
	
	\begin{convention}\label{conv:tree}
		Let $T = T(r,s)$ be a tree with bipartition $(R, S)$. By 
		Observation~\ref{obs:indep_cover}, $T$ has a MIC of size 
		at most $\min(r,s)$. Throughout this paper, we always label 
		the bipartition of $T(r,s)$ so that $|R| = r \leq s = |S|$, 
		and we identify $R$ as the MIC of $T$. That is, we write 
		$T = T(r,s)$ with the convention that $R$ is the smaller 
		bipartition class and serves as the MIC in all subsequent 
		arguments.
	\end{convention}
	\ \\
	The following is a weaker upper bound comparing Theorem~\ref{thm:main}.
	
	\begin{proposition}\label{prop:weaker}
		Let $T = T(r,s)$ be a tree with parts $R$ and $S$ such that $|R| = r \leq s = |S|$, and $m$ and $n$ are integers and $n$ is sufficiently large compared to $m$. Then,
		\[
		\ex(m, n;\, T(r,s)) \leq r\cdot n.
		\]
	\end{proposition}
	
	\begin{proof}
		Let $G = (A, B; E)$ be a bipartite graph with $|A| = m$, $|B| = n$, and
		$e(G) > r \cdot n$, where $n$ is sufficiently large compared to $m$, $r$, $s$. Since $e(G) > r \cdot n$ and $|A| = m$ is fixed, we may apply
		Lemma~\ref{lem:common_nbhd} with $b = r$. This yields a subset $A' \subseteq A$ with $|A'| = r$ and a constant $c > 0$ such that the common neighborhood of $A'$ in $B$ satisfies
		\[
		|N_G[A']| \geq cn.
		\]
		In particular, for $n$ sufficiently large (specifically $cn > s$), we have $|N_G[A']| \geq cn > s$.\\
		
		By Observation~\ref{obs:indep_cover}, the part $R$ of $T(r,s)$ with
		$|R| = r$ is a minimum independent vertex cover of $T$. Now, we construct an embedding $\phi: V(T) \to V(G)$ as follows. Since $|R| = r = |A'|$, define $\phi$ restricted to $R$ as any bijection $\phi\!\upharpoonright_R : R \to A'$. And since $|N_G[A']| \geq cn > s = |S|$, we can choose an injective map $\phi\!\upharpoonright_S : S \to N_G[A']$ such that the images of $R$ and $S$ are disjoint since $A' \subseteq A$ and $N_G[A'] \subseteq B$, and $A \cap B = \emptyset$ in the bipartite graph $G$. To verify that every edge of $T$ is mapped to an edge of $G$, let $uv \in E(T)$ be any edge of $T$. By  Observation~\ref{obs:indep_cover}, every edge of $T$ has one endpoint in $R$ and one in $S$. Without loss of generality, $u \in R$ and $v \in S$. By definition of the common neighborhood $N_G[A'] = \bigcap_{w \in A'} N_G(w)$, every vertex in $N_G[A']$ is adjacent in $G$ to \emph{every} vertex in $A'$. In particular, $\phi(v)$ is adjacent to $\phi(u) \in A'$. Therefore, $\phi(u)\phi(v) \in E(G)$. Since $uv$ was an arbitrary edge of $T$, we conclude that $\phi$ maps every edge of $T$ to an edge of $G$.  Hence $\phi$ is a valid subgraph embedding, and
		$T(r,s) \subseteq G$.\\
		
		We have shown that any bipartite graph $G = (A,B;E)$
		with $|A| = m$, $|B| = n$ (sufficiently large), and $e(G) > rn$ contains $T(r,s)$ as a subgraph. Therefore, any $T(r,s)$-free such graph has at most $rn$ edges, i.e.,
		\[
		\ex(m, n;\, T(r,s)) \leq r \cdot n. \qedhere
		\]
	\end{proof}
	
	\bigskip

	Now, we prove Theorem~\ref{thm:main} by restating it. First we prove two lemmas which will be used in the proof.
	
	\begin{lemma}\label{clm:mindeg}
		Let $T = T(r,s)$ be a tree with bipartition $(R, S)$,
		$|R| = r \leq s = |S|$. Then
		\[
		\min_{x \in R}\, d_T(x)
		\;\leq\; \left\lfloor \frac{s}{r} \right\rfloor + 1.
		\]
	\end{lemma}
	
	\begin{proof}
		Since $T$ is a tree on $r + s$ vertices it has $r + s - 1$
		edges, all of which go between $R$ and $S$. Hence
		$\sum_{x \in R} d_T(x) = r + s - 1$, and by averaging
		\[
		\min_{x \in R}\, d_T(x)
		\;\leq\; \frac{r + s - 1}{r}
		\;=\; \frac{s}{r} + 1 - \frac{1}{r}
		\;\leq\; \left\lfloor \frac{s}{r} \right\rfloor + 1.
		\qedhere
		\]
	\end{proof}
	
	\begin{lemma}\label{clm:threshold}
		Let $G = (A, B; E)$ be a bipartite graph with $|A| = m$ and
		$|B| = n$. If $e(G) \geq (r-1)n$ and
		\begin{equation}\label{eq:nthreshold}
			n \;\geq\; s \cdot (m - r + 2)\binom{m}{r-1},
		\end{equation}
		then Lemma~\ref{lem:common_nbhd} applied with $b = r - 1$
		yields a set $A' \subseteq A$ with $|A'| = r - 1$ whose
		common neighborhood in $B$ satisfies $|N_G[A']| \geq cn$,
		where
		\[
		c \;=\; \frac{1}{(m - r + 2)\dbinom{m}{r-1}},
		\]
		and in particular $|N_G[A']| \geq s$.
	\end{lemma}
	
	\begin{proof}
		Substituting $b = r - 1$, $t = \lceil b \rceil = r - 1$,
		and $a = m$ into the constant $c = \frac{b - t + 1}{(a - t
			+ 1)\binom{a}{t}}$ of Lemma~\ref{lem:common_nbhd} gives
		\[
		c \;=\; \frac{1}{(m - r + 2)\dbinom{m}{r-1}}.
		\]
		The condition $|N_G[A']| \geq cn \geq s$ is then equivalent
		to~\eqref{eq:nthreshold}.
	\end{proof}
	
	\bigskip

	\begin{proof}[Proof of Theorem~\ref{thm:main}]
		
		We consider two cases based on the value of $r$.
		
		\medskip

		Case $r = 1$. Here $R = \{x\}$ is a single vertex and $T(1,s)$ is a star with center $x$ and $s$ leaves. Any $T(1,s)$-free bipartite graph $G = (A,B;E)$ has maximum degree at most $s - 1$ on the $A$-side, since any vertex $v \in A$ with $d_G(v) \geq s$ together with $s$ of its neighbors in $B$ forms a copy of $T(1,s)$ with center mapped to $v$. Hence $e(G) \leq (s-1)m$, giving $\ex(m,n;\,T(1,s)) \leq (s-1)m$. The lower bound $\ex(m,n;\,T(1,s)) \geq (s-1)m$ follows by taking $G$ to be the bipartite graph with $A = \{1,\dots,m\}$, $B = \{1,\dots,n\}$, and with vertex $i \in A$ adjacent to the $s-1$ vertices $\{i, i+1, \dots, i+s-2\} \pmod n$ of $B$ for each $i = 1,\dots,m$. Every vertex of $A$ has degree exactly $s-1$. Since an embedding of $T(1,s)$ may center on either side of $G$, we must also bound degrees on the $B$-side. As $n$ is large enough comparing $m$, the reduction map $i \mapsto i \bmod n$ is injective on $\{1,\dots,m\}$. Hence for each $j \in B$, at most one $i \in A$ can be congruent to any given residue, so at most one $i \in A$ is adjacent to $j$ for each of the $s-1$ residues $j, j-1, \dots, j-s+2 \pmod n$, giving $d_G(j) \leq s-1$. Hence $\Delta(G) \le s-1$ on both sides, so $G$ contains no copy of $T(1,s)$ centered on either side, and $G$ is $T(1,s)$-free. As $e(G) = (s-1)m$ (each of the $m$ vertices of $A$ contributing exactly $s-1$ distinct edges), we conclude $\ex(m,n;\,T(1,s)) \geq (s-1)m$. Hence $\ex(m,n;\,T(1,s)) = (s-1)m$.
		
		\medskip
		Case $r \geq 2$.
		Let $G = (A,B;E)$ be a $T(r,s)$-free bipartite graph with $|A| = m$, $|B| = n$ satisfying the threshold condition. We show $e(G) \leq (r-1)n + C(m,r,s)$ directly by edge counting, without assuming any lower bound on $e(G)$.
		
		\medskip
		
		We construct a sequence of pairs $(A_1, B_1), (A_2, B_2), \ldots$ where each $A_t \subseteq A$ has $|A_t| = r - 1$,
		$B_t \subseteq B$ is the common neighborhood of $A_t$ in the residual part of $B$, and $B_1, B_2, \ldots$ are
		pairwise disjoint. Set $B^{(0)}_{\rm res} = B$. At step $t \geq 1$, let
		$B^{(t-1)}_{\rm res} = B \setminus \bigcup_{i < t} B_i$.
		If
		\begin{equation}\label{eq:largeres}
			\left|B^{(t-1)}_{\rm res}\right|
			\;\geq\; s \cdot (m - r + 2)\binom{m}{r-1},
		\end{equation}
		then by Lemma~\ref{clm:threshold} applied to $G[A,
		B^{(t-1)}_{\rm res}]$,
		there exists $A_t \subseteq A$ with $|A_t| = r - 1$ since $e(A, B^{(t-1)}_{\rm res})
		\geq (r-1)|B^{(t-1)}_{\rm res}|$. And
		\[
		B_t \;:=\; N_G[A_t] \cap B^{(t-1)}_{\rm res},
		\quad |B_t| \;\geq\; c\,|B^{(t-1)}_{\rm res}|,
		\quad |B_t| \;\geq\; s.
		\]
		We then set $B^{(t)}_{\rm res} = B^{(t-1)}_{\rm res}
		\setminus B_t$ and proceed to step $t + 1$.
		
		The iteration terminates when~\eqref{eq:largeres} fails,
		leaving a residual $B_{\rm res} = B^{(T_{\max})}_{\rm res}$
		with
		\[
		|B_{\rm res}| \;<\; s \cdot (m - r + 2)\binom{m}{r-1}.
		\]
		For convenience, we set $ \kappa \,=\, s \cdot (m - r + 2)\binom{m}{r-1}
		,$ where $\kappa$ is a constant depending only on $m$, $r$, $s$.\\
		
		We show that the sets $A_t$ are pairwise disjoint. Suppose first that $A_i = A_j$ for some $i < j$. At step $j$, the lemma is applied to the residual $B^{(j-1)}_{\rm res} = B \setminus \bigcup_{\ell < j} B_\ell$, which by construction is disjoint from $B_i$. However, $B_j \subseteq N_G[A_j] = N_G[A_i]$, and at step $i$ the set $B_i$ was defined as the common neighborhood of $A_i$ in the residual at that point, so $N_G[A_i] \cap B^{(j-1)}_{\rm res} = \emptyset$, meaning no common neighborhood of $A_j = A_i$ exists in $B^{(j-1)}_{\rm res}$, contradicting the application of Lemma~\ref{lem:common_nbhd} at step $j$.\\ 
		
		Suppose next that $A_i \neq A_j$ but $A_i \cap A_j \neq \emptyset$ for some $i \neq j$. Let $x \in A_i \cap A_j$ and $y \in A_j \setminus A_i$. Since $x \in A_i$, the vertex $x$ is adjacent to every vertex of $B_i$ in $G$, and since $x \in A_j$, it is also adjacent to every vertex of $B_j$ in $G$. We embed $T$ into $G$ as follows. Label the vertices of $R$ as $\rho_1, \ldots, \rho_r$. Map $\rho_1, \ldots, \rho_{r-2}$ injectively into $A_i \setminus \{x\}$, which has size $r - 2$, and map the $S$-neighbors of each $\rho_\ell$ for $\ell \leq r - 2$ injectively into $B_i$, which is possible since $|B_i| \geq s$. Map $\rho_{r-1} \mapsto x$ and the $S$-neighbors of $\rho_{r-1}$ injectively into $B_j$, which is possible since $x \in A_j$ implies $x$ is adjacent to all of $B_j$ and $|B_j| \geq s$. Map $\rho_r \mapsto y$ and the $S$-neighbors of $\rho_r$ injectively into $B_j \setminus \{\text{used vertices}\}$, which is possible since $y \in A_j$ implies $y$ is adjacent to all of $B_j$ and $|B_j| \geq s$ with $n$ sufficiently large. Since $B_i$ and $B_j$ are disjoint with $|B_i|, |B_j| \geq s$, all vertices of $S$ can be mapped injectively, and the $r$ vertices of $R$ map to the distinct vertices $(A_i \setminus \{x\}) \cup \{x, y\}$ in $A$, giving a valid embedding of $T$ into $G$ and contradicting $T(r,s)$-freeness. Hence, all sets $A_t$ are pairwise disjoint. Since $\sum_t |A_t| = (r-1) \cdot |\{t\}| \leq |A| = m$, the iteration runs for at most $\lfloor m/(r-1) \rfloor$ steps.\\

		Now, we show the degree bound into each $B_t$. We claim that for each $t$ and each $v \in A \setminus A_t$,
		\begin{equation}\label{eq:degbound}
			d_G(v,\, B_t) \;\leq\; \left\lfloor \frac{s}{r} \right\rfloor.
		\end{equation}
		Suppose not, so $d_G(v, B_t) \geq \lfloor s/r \rfloor + 1$ for some $v \in A \setminus A_t$. By Lemma~\ref{clm:mindeg}, pick $\rho \in R$ with $d_T(\rho) \leq \lfloor s/r \rfloor + 1 \leq d_G(v, B_t)$. Embed $T$ into $G$ as follows: map $\rho \mapsto v$ and the $d_T(\rho)$ neighbors of $\rho$ in $S$ injectively into $N_G(v) \cap B_t$, map the remaining $r - 1$ vertices of $R \setminus \{\rho\}$ injectively into $A_t$, and map all remaining vertices of $S$ injectively into $B_t \setminus \{\text{used vertices}\}$, which is possible since every vertex of $B_t$ is adjacent to all of $A_t$ and $|B_t| \geq s$. This gives a copy of $T$ in $G$, a contradiction. Hence~\eqref{eq:degbound} holds. As a consequence, since all $A_t$ are pairwise disjoint, every $u \in A_j$ satisfies $u \in A \setminus A_i$ for any $i \neq j$, and therefore $d_G(u, B_i) \leq \lfloor s/r \rfloor$ by~\eqref{eq:degbound}. Hence,
		\begin{equation}\label{eq:cross}
			e(A_j,\, B_i) \;\leq\; (r-1)\left\lfloor \frac{s}{r} \right\rfloor \;\leq\; s - 1 \;=\; O(1)
		\end{equation}
		for every pair $i \neq j$.\\
		
		Now, we count the edges in $G$.	Decompose the edges of $G$ as
		\begin{equation}\label{eq:decomp}
			e(G) \;=\;
			\underbrace{\sum_{t} e(A_t,\, B_t)}_{\rm (I)}
			\;+\;
			\underbrace{\sum_{i \neq j} e(A_i,\, B_j)}_{\rm (II)}
			\;+\;
			\underbrace{e\!\left(A \setminus \bigcup_t A_t,\;
				\bigcup_t B_t\right)}_{\rm (III)}
			\;+\;
			\underbrace{e(A,\, B_{\rm res})}_{\rm (IV)}.
		\end{equation}
		
		\begin{itemize}[leftmargin=2em, label={}]
			
			\item Term~(I). Since $|A_t| = r - 1$ and
			the sets $B_t$ are pairwise disjoint subsets of $B$,
			\[
			\sum_t e(A_t,\, B_t)
			\;\leq\; (r - 1)\sum_t |B_t|
			\;\leq\; (r - 1)\,n.
			\]
			
			\item Term~(II). By~\eqref{eq:cross}, each
			ordered pair $(i, j)$ with $i \neq j$ contributes at
			most $s - 1$ edges. With $T_{\max} \leq \binom{m}{r-1}$
			pairs,
			\[
			\sum_{i \neq j} e(A_i,\, B_j)
			\;\leq\; \binom{m}{r-1}^2 \cdot (s - 1).
			\]
			
			\item Term~(III). Each vertex $v \in A
			\setminus \bigcup_t A_t$ satisfies $d_G(v, B_t) \leq
			\lfloor s/r \rfloor$ for every $t$ by~\eqref{eq:degbound}.
			Hence
			\[
			e\!\left(A \setminus \bigcup_t A_t,\;
			\bigcup_t B_t\right)
			\;\leq\; m \cdot T_{\max} \cdot
			\left\lfloor \frac{s}{r} \right\rfloor
			\;\leq\; m \cdot \binom{m}{r-1} \cdot s.
			\]
			
			\item Term (IV). Since $|B_{\mathrm{res}}| < \kappa$,
			\[e(A,B_{\mathrm{res}}) \le m\kappa\].
			
		\end{itemize}
		
		Combining all four terms in (5), and since $m,r,s$ are fixed constants
		as $n \to \infty$, terms (II)--(IV) are each $O(1)$ in $n$ and sum to the
		constant $C(m,r,s)$ of Theorem~1.2. Hence $e(G) \le (r-1)n + C(m,r,s)$,
		and $\ex(m,n;T(r,s)) \le (r-1)n + C(m,r,s)$.\qedhere
	\end{proof}

	\bigskip
	
	Now, we prove the Theorem~\ref{thm:starforest}.
	
	\begin{proof}[Proof of Theorem~\ref{thm:starforest}]
		The lower bound $\ex(m, n; F) \geq e(\mathcal{B}(m, n; k))$ follows from
		the construction of  $\mathcal{B}(m, n; k)$. We prove the upper bound by induction on $k$. For the base case, let $k = 1$. Here $F = S_{d_1}$ is a single star. An $S_{d_1}$-free bipartite graph $G$ has maximum degree at most $d_1-1$ on the $A$-side, since any vertex $v \in A$ with $d_G(v) \ge d_1$ together with $d_1$ neighbors in $B$ forms a copy of $S_{d_1}$, so $e(G) \le (d_1-1)m$, and this is achieved by $\mathcal{B}(m,n;1)$ since its pairwise disjoint neighborhoods give every vertex of $B$ degree at most $1 \le d_1-1$, making it $S_{d_1}$-free as well, hence $e(G) \le (d_1-1)m = e(\mathcal{B}(m,n;1))$.Now, assume the theorem holds for all star forests
		with fewer than $k$ components. Let $G = (A, B; E)$ be an extremal $F$-free
		bipartite graph with $|A| = m$, $|B| = n$, and $e(G) = \ex(m, n; F)$. Let $F' = F - S^k = \bigcup_{i=1}^{k-1} S_{d_i}$. By the induction hypothesis,
		\[
		\ex(m, n;\, F') \;=\; (k-2)\,n \;+\; (d_{k-1} - 1)(m - k + 2).
		\]
		Since $d_{k-1} \geq d_k$ and $n$ is sufficiently large compared to $m$,
		\begin{align*}
			e(\mathcal{B}(m,n;k)) - \ex(m, n; F')
			&= (k-1)n + (d_k-1)(m-k+1) \\
			&\quad - (k-2)n - (d_{k-1}-1)(m-k+2) \\
			&= n + (d_k - 1)(m-k+1) - (d_{k-1}-1)(m-k+2) \\
			&\geq n - (d_{k-1} - 1)(m - k + 2) \;>\; 0,
		\end{align*}
		for $n$ sufficiently large compared to $m$. Hence $e(G) \geq
		e(\mathcal{B}(m,n;k)) > \ex(m, n; F')$, and by the induction hypothesis
		$F' \subseteq G$. That is, there exist $k-1$ vertex-disjoint stars
		$S^1, \ldots, S^{k-1}$ in $G$. Fix any $j \in \{1, \ldots, k-1\}$ and let $S^j$ be the $j$-th star in the copy
		of $F'$ found above, with center $c_j$ and leaf set $L_j$. Since $G$ is $F$-free,
		$G - S^j$ must be $(F - S^j)$-free. Note that $F - S^j$ is a star forest with
		$k - 1$ components with $S^j$ removed and the remaining stars unchanged. By the
		induction hypothesis applied to $F - S^j$,
		\[
		e(G - S^j) \;\leq\; \ex(m, n;\, F - S^j)
		\;=\; (k-2)\,n \;+\; O(m).
		\]
		Let $e_0$ denote the number of edges between $V(S^j)$ and $V(G) \setminus
		V(S^j)$. Then
		\begin{align*}
			e_0 &= e(G) - e(G[S^j]) - e(G - S^j) \\
			&\geq (k-1)n + (d_k-1)(m-k+1) - \binom{d_j+1}{2} - (k-2)n - O(m) \\
			&= n + (d_k - 1)(m-k+1) - O(m) \\
			&= \Omega(n).
		\end{align*}
		
		Since $|V(S^j)| = d_j + 1$ is a fixed constant and $e_0 = \Omega(n)$, 
		at least one vertex $u_j \in V(S^j)$ has degree $\Omega(n)$ in $G$. 
		We call $u_j$ the \emph{linear-degree vertex} of $S^j$. Note that 
		$u_j \in A$: if the center $c_j \in A$ then $e_0$ concentrates at 
		$c_j$ since leaves $L_j \subseteq B$ each have degree at most $m = 
		O(m) \ll n$ into $A$; if $c_j \in B$ then leaves $L_j \subseteq A$ 
		and at least one leaf $\ell^* \in A$ satisfies $d_G(\ell^*) = \Omega(n)$, 
		so we take $u_j = \ell^*$. In either case, $u_j \in A$ has linear 
		degree and can serve as a star center in $A$. \\
		
		Since the above holds for every $j \in \{1, \ldots, k-1\}$, let
		$ U \;=\; \{c_1, c_2, \ldots, c_{k-1}\} $
		be the set of $k-1$ linear-degree centers. We claim $G - U$ is $S_{d_k}$-free. Otherwise, a copy of $S_{d_k}$ 
		in $G - U$ together with stars centered at $u_1, \ldots, u_{k-1}$ 
		(since each $u_j$ has linear degree $\Omega(n)$) yields a copy of $F$ in $G$, contradicting $F$-freeness. Hence every vertex of $A \setminus U$ has degree at most $d_k - 1$ into $B$, and therefore $G \subseteq \mathcal{B}(m, n; k)$. Therefore $G \subseteq \mathcal{B}(m, n; k)$, giving
		\[
		e(G) \;\leq\; e(\mathcal{B}(m, n; k))
		\;=\; (k-1)\,n \;+\; (d_k - 1)(m - k + 1).
		\]
		This completes the induction and the proof of Theorem~\ref{thm:starforest}.
	\end{proof}
	
	\begin{remark}
		By the results of Lidick\'{y}--Liu--Palmer~\cite{LLP2013}, in the general-host result, the leading coefficient of $n$ in $e(F(n,i))$ is
		$f_i = i - 1 + \frac{d_i - 1}{2}$, which depends on both $i$ and $d_i$,
		creating a non-trivial competition between indices and requiring three separate
		cases. In our bipartite setting, the leading coefficient is simply $i - 1$, so the
		maximum is always achieved at $i = k$ and the proof reduces to a single
		inductive case.
	\end{remark}

	
	
	\section*{Acknowledgments}
	The author thanks Professor Nuh Aydin and Lucas Waite for helpful conversations related to bipartite Tur\'an numbers for trees and for
	bringing this problem to the author's attention. Independently, and using a different approach, they obtained related results, now
	available in \cite{WaiteAydin}.

\end{document}